\documentclass{article}
\usepackage{amsmath, amsfonts, amssymb, amsthm}
\usepackage[top=1in, bottom=1in]{geometry}

    \usepackage[T1]{fontenc}
    \usepackage{ae}
    \usepackage{aecompl}

\newtheorem{theorem}{Theorem}
\newtheorem{lemma}[theorem]{Lemma}
\newtheorem{proposition}[theorem]{Proposition}

\newtheorem{conjecture}{Conjecture}

\newcommand{\myRef}[1]{#1}
\newcommand{\mySection}[1]{Section #1}

\newcommand{\R}{\mathbb{R}}
\newcommand{\C}{\mathbb{C}}
\newcommand{\K}{\mathbb{K}}

\newcommand{\E}{\mathbb{E}}

\newcommand{\Pee}{\mathcal{P}}
\newcommand{\otherB}{\tilde{B}}

\newcommand{\dd}[1]{\,\textit{d}#1}

\newcommand{\pr}{\mathbb{P}}
\newcommand{\perm}[1]{\text{per}(#1)}

\begin{document}

\title{A stability result using the matrix norm to bound the permanent}
\author{Ross Berkowitz\footnote{Rutgers University.  Email: \texttt{rkb73@math.rutgers.edu}} \qquad \qquad Pat Devlin\footnote{Rutgers University.  Email: \texttt{prd41@math.rutgers.edu}.  Supported by NSF grant DMS1501962.}}
\date{June 23, 2016}

\maketitle
\begin{abstract}
We prove a stability version of a general result that bounds the permanent of a matrix in terms of its operator norm.  More specifically, suppose $A$ is an $n \times n$ matrix over $\C$ (resp. $\R$), and let $\Pee$ denote the set of $n \times n$ matrices over $\C$ (resp. $\R$) that can be written as a permutation matrix times a unitary diagonal matrix.  Then it is known that the permanent of $A$ satisfies $|\perm{A}| \leq \Vert A \Vert_{2} ^n$ with equality iff $A/ \Vert A \Vert_{2} \in \Pee$ (where $\Vert A \Vert_2$ is the operator $2$-norm of $A$).  We show a stability version of this result asserting that unless $A$ is very close (in a particular sense) to one of these extremal matrices, its permanent is exponentially smaller (as a function of $n$) than $\Vert A \Vert_2 ^n$.  In particular, for any fixed $\alpha, \beta > 0$, we show that $|\perm{A}|$ is exponentially smaller than $\Vert A \Vert_2 ^n$ unless all but at most $\alpha n$ rows contain entries of modulus at least $\Vert A \Vert_2 (1 - \beta)$.
\end{abstract}

\section{Introduction}\label{section intro}
The \textit{permanent} of an $n \times n$ matrix, $A$, has long been an important quantity in combinatorics and computer science, and more recently it has also had applications to physics and linear-optical quantum computing.  It is defined as
\[
\perm{A} := \sum_{\sigma\in S_n} \prod_{i=1} ^{n} a_{i, \sigma(i)},
\]
where $S_n$ denotes the set of permutations of $[n]=\{1,2,\ldots,n\}$.  For instance, if $A$ only has entries in $\{0,1\} \subseteq \R$, then the permanent counts the number of perfect matchings in the bipartite graph whose bipartite adjacency matrix is $A$.

\paragraph*{}The definition of the permanent is of course reminiscent of that for the determinant; however, whereas the determinant is rich in algebraic and geometric meaning, the more combinatorial permanent is notoriously difficult to understand.  For example, computing $\perm{A}$ even for $\{0, 1\}$-matrices is the prototypical \#P-complete problem (Valiant \cite{valiant}).

\paragraph*{}On the other hand, the \textit{operator 2-norm} (also called the \textit{operator norm}) of a matrix is a particularly nice parameter.  For an $n \times n$ matrix $A$ with entries in $\C$, it is defined as
\[
\Vert A \Vert_{2} = \sup_{\Vert \vec{x} \Vert_{2} \leq 1, \ \vec{x} \in \C^n} \Vert A \vec{x} \Vert_{2},
\]
where $\Vert \vec{v} \Vert_{p}$ is the usual $l_{p}$ norm (i.e.,  $\Vert \vec{v} \Vert_p ^{p} = \sum_{i} |v_i|^{p}$ for $p \in (0, \infty)$, and $\Vert \vec{v} \Vert_{\infty} = \max |v_i|$).  The operator norm of a matrix has the advantages of being both algebraically and analytically well-behaved as well as computationally easy to determine (as this amounts to finding the largest singular value of $A$).

\paragraph*{}Considering how differently behaved the permanent and operator norm are, it is perhaps strange to think that there would be much of a connection between them.  Nonetheless, they are related by the following extremal result, which is due to Gurvits \cite{gurvits} (see also \cite{aaronson, nguyen}).

\begin{theorem}\label{norm result}
Suppose $A$ is an $n \times n$ matrix over $\C$ (resp. $\R$), and let $\Pee$ denote the set of $n \times n$ matrices over $\C$ (resp. $\R$) that can be written as a permutation matrix times a unitary diagonal matrix.  Then $|\perm{A}| \leq \Vert A \Vert_{2} ^n$ with equality iff $A$ is a scalar multiple of a matrix in $\Pee$.
\end{theorem}
\paragraph*{}Note that this extremal set $\Pee$ is simply the set of matrices with exactly $n$ non-zero entries, each having modulus 1, and no two of which are in the same row or column.  Such a matrix $P \in \Pee$ has $\Vert P \Vert_2 = | \perm{P}| = 1$ and satisfies
\[
\Vert AP \Vert_2 = \Vert PA \Vert_2 = \Vert A \Vert_2, \qquad \qquad \text{and} \qquad \qquad |\perm{AP}| = |\perm{PA}| = | \perm{A}|
\]
for all matrices $A$ (which is equivalent to membership in $\Pee$).  Moreover, $\Pee$ is a subgroup of the group of unitary matrices, and as a set, it has a very tractable topological structure.

\paragraph*{}Motivated by algorithmic questions related to approximating the permanent, Aaronson and Hance \cite{aaronson} asked whether one could prove a stability version of \myRef{Theorem \ref{norm result}}:
\paragraph*{Question A:} If $|\perm{A}|$ is close to $\Vert A \Vert_2 ^n$, must $A / \Vert A \Vert_2$ be `close' to a matrix in $\Pee$?

\paragraph*{}A somewhat more concrete version was suggested by Aaronson and Nguyen \cite{nguyen}:
\paragraph*{Question B:} Characterize $n \times n$ matrices $A$ such that $\Vert A \Vert_2 \leq 1$ and there exists a constant $C > 0$ such that $|\perm{A}| \geq n^{-C}$.

\paragraph*{}Using techniques of inverse Littlewood-Offord theory, Aaronson and Nguyen gave a substantial answer to an analogous question under the (stronger) assumptions that $A$ is orthogonal and that the intersection of the hypercube $\{\pm 1\}^n$ with its image under $A$ is large.  They also proved something like (actually slightly stronger than) our results below for stochastic matrices.  Further results in the direction of \myRef{Question B} were given by Nguyen \cite{nguyenPrivate}.

\paragraph*{}The two main results of the present paper are \myRef{Theorems \ref{main theorem}} and \myRef{\ref{main theorem real case}} below.  The first provides a positive answer to \myRef{Question A} for matrices over $\C$ (or $\R$), and the second is a more refined result that (depending on your philosophical views) at least partially addresses \myRef{Question B} for matrices over $\R$.  More specifically, we bound $\perm{A}$ in terms of the following easily computed parameters.

\paragraph*{Definition:} Let $A$ be a matrix with rows $r_1, r_2, \ldots , r_n$, and $p \in \R \cup \{\infty\}$.  Then the parameter $h_p (A)$ is defined as $h_{p} (A) = h_p = \frac{1}{n} \sum_{i} \Vert r_i \Vert_{p}$.

\paragraph*{}We will only consider $h_\infty$ and $h_2$.  First note $0 \leq h_{\infty} (A) \leq h_{2} (A) \leq \Vert A \Vert_2$.  Moreover, it is easy to show $h_{2} (A) = \Vert A \Vert_2$ iff $A/\Vert A \Vert_2$ is a unitary matrix, and $h_{\infty} (A) = \Vert A \Vert_2$ iff $A / \Vert A \Vert_2$ is in $\Pee$.  Thus, in some sense, the quantity $1- h_{2}(A) / \Vert A \Vert_2 \in [0, 1]$ measures how close $A / \Vert A \Vert_2$ is to being unitary, and $1 - h_{\infty}(A) / \Vert A \Vert_2 \in [0, 1]$ measures how close $A / \Vert A \Vert_2$ is to being in $\Pee$.  Broadly speaking, $h_{\infty} / \Vert A \Vert_2$ is close to 1 precisely when most of the rows of $A$ each have one entry of modulus close to $\Vert A \Vert_2$ and all the other entries in that row are close to 0.

\paragraph*{}Before stating the first of our main results, notice that in addressing either of the above questions, we lose no generality in assuming $\Vert A \Vert _2 \leq 1$, since \myRef{Question A} is invariant under scaling.  However, to facilitate any application of our results, we state them in the ``more general" case that $\Vert A \Vert_2 \leq T$.

\newpage
\begin{theorem}\label{main theorem}
Let $A$ be an $n\times n$ matrix over $\C$ and $\Vert A \Vert_{2} \leq T \neq 0$.  Then
\begin{itemize}
\item[(i)] $|\perm{A}| \leq 2T^{n} \exp \bigg[-3n\Big(1- \frac{\sqrt{\pi}}{2} h_{2}/T - \left(1- \frac{\sqrt{\pi}}{2} \right)h_{\infty} /T \Big)^2/100 \bigg],$
\item[(ii)] $|\perm{A}| \leq 2 T^{n} \exp[-n(1-h_{\infty} / T)^2 /10^{5}]$.
\end{itemize}
\end{theorem}

\paragraph*{}As discussed above, this provides a positive answer to \myRef{Question A} by viewing $h_\infty$ (and to a lesser extent $h_2$) as a proxy for `closeness' of a matrix $A$ to those in $\Pee$.  As an easy corollary, if $\alpha, \beta \geq 0$ satisfy $|\perm{A}| \geq 2 T^{n} \exp[-n \alpha^2 \beta^2 / 10^5]$, then all but at most $\alpha n$ of the rows of $A$ contain an entry whose modulus is at least $T(1-\beta)$.  And since the $l_2$ norm of any row of $A$ is at most $\Vert A \Vert_2$, no entry of $A$ can have modulus larger than $T$.  Thus, entries of modulus $T(1-\beta)$ are nearly as large as possible.  Moreover, if a row (or column) has an entry with very large modulus, then the remaining entries must have very small moduli (again since its $l_2$ norm is at most $\Vert A \Vert_2$).  Thus, this theorem also provides a \textit{qualitative} stability result stating that matrices with large permanent must have many very large entries, and a row (or column) containing a large entry must have all its other entries small.

\paragraph*{}Note that \myRef{Theorem \ref{main theorem}} is only useful for values of $h_{\infty} / T$ that are not very close to $1$---namely when $1 - h_{\infty} / T \gg n^{-1/2}$.  Although this does well in many cases, we believe that for large values of $h_\infty / T$, it is not optimal.  For comparison, if $A$ is $\delta$ times the identity matrix, and $\delta \approx 1$, then $|\perm{A}| \approx e^{-n (1-\delta)} = e^{-n (1-h_{\infty}) }$, and we conjecture that this is essentially tight.

\begin{conjecture}\label{main conjecture}
There is some constant $C > 0$ and some polynomial $f(n)$ such that the following holds.  If $A$ is an $n \times n$ matrix with complex entries and $\Vert A \Vert_2 \leq 1$, then $| \perm{A}| \leq f(n) e^{-Cn(1 - h_\infty)}$.
\end{conjecture}

\paragraph*{}As a step in this direction, we are able to prove the following, which better addresses \myRef{Question B} for matrices over $\R$.

\begin{theorem}\label{main theorem real case}
Let $A$ be an $n\times n$ matrix over $\R$ and $\Vert A \Vert_{2} \leq T \neq 0$.  Then
\[
|\perm{A}| \leq T^{n} (n+6) \exp \left[ \dfrac{-\sqrt{n(1-h_{\infty} / T)}}{400} \right].
\]
\end{theorem}

\paragraph*{}As with \myRef{Theorem \ref{main theorem}}, a result like \myRef{Theorem \ref{main theorem real case}} that involves $h_{2}$ is also possible, and it essentially falls out of our proof directly.  \myRef{Theorem \ref{main theorem real case}} is an improvement over \myRef{Theorem \ref{main theorem}} when $n^{-1/3} \gg 1- h_{\infty} / T$ and gives a meaningful bound provided $1-h_{\infty} / T \gg \log(n)^2 / n$.  Although this yields a quantitatively better understanding for matrices over $\R$, we cannot shake the belief that neither of our main results (i.e., \myRef{Theorems \ref{main theorem}} and \myRef{\ref{main theorem real case}}) is best possible, and we discuss this further in \myRef{Section \ref{section conclusion}}.

\subsection*{Structure of paper}
The paper is devoted to proving \myRef{Theorems \ref{main theorem}} and \myRef{\ref{main theorem real case}}, which goes roughly as follows.  First, we appeal to a result of Glynn \cite{glynn} that allows us to convert the problem of estimating the permanent into a problem about estimating the expected value of a certain random variable (\mySection{\ref{section set-up}}).  We then use standard probabilistic tools to show certain concentration results for the random variable of interest, which in turn yield the estimates needed for our results.  This is done for the complex-valued case in \mySection{\ref{section complex}}, which proves \myRef{Theorem \ref{main theorem}}.  In \mySection{\ref{section real}}, we consider the real-valued case, where we analyze the corresponding random variable more carefully to obtain \myRef{Theorem \ref{main theorem real case}}.  We conclude in \mySection{\ref{section conclusion}} with several open questions and conjectures, as well as a discussion of \myRef{Question B}.

\section{Definitions and set-up with random variables}\label{section set-up}
We first need to use an observation due to Glynn \cite{glynn} whereby the permanent of a matrix is expressed as the expectation of a certain random variable.  We will work over the field $\K$, which will either be $\R$ or $\C$.

\paragraph*{}Given an $n \times n$ matrix $A$ over $\K$ and $x \in \K^n$, set $y = Ax$, and define the \textit{Glynn estimator} of $A$ at $x$ to be
\[
Gly_{x} (A) = \prod_{i=1} ^{n} \overline{x}_i  \times \prod_{i=1} ^{n} y_i,
\]
where $\overline{z}$ denotes the complex conjugate of $z$.  Let $X \in \K^n$ be the random variable whose coordinates are independently selected uniformly on $|z| = 1$, and let $Y = AX$ (note: if $\K = \C$, then each coordinate of $X$ is distributed continuously over the unit circle, whereas if $\K= \R$, then $X$ is chosen uniformly from the discrete set $\{-1,1\}^n$).  Then
\[
\perm{A} = \E [Gly_{X} (A)] = \E \left[ \prod_{i=1} ^{n} \overline{X}_i Y_i \right],
\]
obtained simply by expanding out the product in the Glynn estimator and using the fact that the $X_i$ are independent with mean $0$ and variance $1$ (see the original proof due to Glynn \cite{glynn} or also \cite{gurvits, aaronson, nguyen}).  Therefore, by convexity (which we are about to use twice), we have
\[
| \perm{A} | \leq \E \left[ \prod_{i=1} ^{n} | \overline{X}_i Y_i| \right] = \E \left[ \prod_{i=1} ^{n} |Y_i| \right] \leq \E \left[ \left( \dfrac{1}{n}\sum_{i=1} ^{n} |Y_i| \right) ^{n} \right] = \E \left[ \left( \dfrac{\Vert AX \Vert_{1}}{n} \right) ^{n} \right].
\]
Note that from here, we could say (by Cauchy-Schwartz)
\[
\dfrac{\Vert AX \Vert_{1}}{n} \leq \dfrac{\Vert AX \Vert_{2}}{\sqrt{n}} = \dfrac{\Vert AX \Vert_{2}}{\Vert X \Vert_{2}} \leq \Vert A \Vert_{2},
\]
thus obtaining the inequality $| \perm{A} | \leq \Vert A \Vert_2 ^n$ of \myRef{Theorem \ref{norm result}} (the equality case follows by considering equality in the above estimates).

\subsection*{Specializing to norm at most 1}
Note that to prove our results, it suffices to prove them for the case $\Vert A \Vert_{2} \leq 1$.  This is because otherwise, we could simply scale the matrix by some $\alpha$ to have norm at most 1, and because $\perm{A} = \alpha^n \perm{A/ \alpha}$, our results would follow.  As such, we will henceforth assume $\Vert A \Vert_2 \leq 1$ (explicitly making note of when we do), but this choice is simply for notational ease.  We remark that the set-up thus far has also been employed in several other papers \cite{gurvits, aaronson, nguyen}; however, the remainder of this paper deviates from the previous literature.

\section{Proof of \myRef{Theorem \ref{main theorem}} ($\K = \C$)}\label{section complex}
In the setting where $\Vert A \Vert_2 \leq 1$, the permanent is always bounded above by $1$ (as shown above), and we want to conclude that under certain conditions, it must be (exponentially) small.  We know (since $0 \leq \Vert A X \Vert_1 /n \leq \Vert A \Vert_{2} \leq 1$) that for all $\varepsilon \geq 0$ and all $\tilde{\mu} \geq 0$,
\[
|\perm{A}| \leq \E \left[ \left( \dfrac{\Vert AX \Vert_{1}}{n} \right) ^{n} \right] \leq (\tilde{\mu} /n + \varepsilon)^n + \pr (\Vert A X \Vert_1 \geq \tilde{\mu} +  \varepsilon n).
\]
We will pick $\tilde{\mu}$ suitably small with $\tilde{\mu} \geq \E[ \Vert AX \Vert_1 ]$ and then argue that $\Vert A X \Vert_1$ is tightly concentrated about its mean, which will complete the proof.

\subsection*{The mean of $\Vert AX\Vert_1$}
We appeal to a theorem of K\"onig, Sch\"utt, and Tomczak-Jaegermann \cite{konig}, which is a variant of Khintchine's inequality conveniently well-suited for our situation (in fact, $X$ was chosen in part so that we could apply this result directly).
\begin{theorem}[K\"onig et al.\ \cite{konig}, $1999$] \label{konig_ineq}
Let $\K$ be $\R$ or $\C$.  Suppose $\vec{a} = (a_1, \ldots , a_n) \in \K^n$ is fixed, and suppose each coordinate of $\xi \in \K^n$ is independently distributed uniformly on $|z| = 1$.  Then
\[
\left| \E \left[ \left| \sum_{i} a_i \xi_i \right| \right] - \Lambda_{\K} \Vert \vec{a} \Vert_{2} \right| \leq \left(1 - \Lambda_{\K} \right) \Vert \vec{a} \Vert_{\infty},
\]
where $\Lambda_{\R} = \sqrt{2/ \pi}$ and $\Lambda_{\C} = \sqrt{\pi}/2$.
\end{theorem}
Applying this to each row of $A$ (and using linearity of expectation) gives

\begin{proposition}\label{mean bound}
With $A$ and $X \in \C^n$ as in \myRef{Section \ref{section set-up}}, we have
\[
\E[\Vert A X \Vert_1 /n] \leq \dfrac{1}{n} \sum_{i=1} ^{n} \left[\sqrt{\pi}/2 \Vert r_i \Vert_2 + \left(1 - \sqrt{\pi}/2 \right) \Vert r_i \Vert_{\infty} \right] =  \dfrac{\sqrt{\pi}}{2} h_{2}(A) + \left(1 - \dfrac{\sqrt{\pi}}{2} \right)  h_{\infty}(A).
\]
\end{proposition}

\subsection*{Concentration about mean}
To show concentration of $\Vert A X\Vert_1$ about its mean, we use a very general and useful result of Talagrand (a form of ``Talagrand's inequality"), which can be found in chapter 1 of his book \cite{talagrand}.
\begin{theorem}[Talagrand \cite{talagrand}, $1991$]\label{talagrand gaussian theorem}
Suppose $f : \R^n \to \R$ is such that $|f(x) - f(y)| \leq \sigma \Vert x - y \Vert_2$ for all $x, y \in \R^n$, and define the random variable $F = f( \xi_1, \xi_2, \ldots, \xi_n)$, where the $\xi_i$ are independent standard normal random variables.  Then for all $t \geq 0$, 
\[
\pr(F > \E[F] +t) \leq e^{-2 t^2 / (\pi \sigma)^2}.
\]
\end{theorem}
We apply this result to our setting by way of a now standard trick that expresses our random variable of interest as a function of standard Gaussians.  In fact, this trick is even discussed in \cite{talagrand}, so we could have saved a few lines of the following argument by simply citing a ``more applicable" version of \myRef{Theorem \ref{talagrand gaussian theorem}} (i.e., one for which this trick has already been incorporated); however, the trick so nicely captures the usefulness of \myRef{Theorem \ref{talagrand gaussian theorem}}, that we thought it worth recalling here.
\begin{proposition}\label{our concentration}
Suppose $\Vert A \Vert_2 \leq 1$, and let $X \in \C^n$ be as in \myRef{Section \ref{section set-up}}.  Then for all $t \geq 0$,
\[
\pr(\Vert AX \Vert_1 > \E[\Vert AX \Vert_1] + tn) \leq e^{- n t^2 / \pi^3}.
\]
\end{proposition}
\begin{proof}
To make use of \myRef{Theorem \ref{talagrand gaussian theorem}}, we need to define a suitable $f : \R^n \to \R$, which we do in pieces.  First define $\Phi : \R \to \R$ via
\[
\Phi(u) = \dfrac{1}{\sqrt{2 \pi}} \int_{-\infty} ^{u} e^{-x ^2 / 2} \dd{x},
\]
which is the probability that a standard Gaussian is at most $u$.  Then define $g : \R^n \to \C^n$ as
\[
g(x_1, \ldots , x_n) = \begin{pmatrix}
e^{2 \pi i \Phi(x_1)}\\
e^{2 \pi i \Phi(x_2)}\\
\vdots \\
e^{2 \pi i \Phi(x_n)}
\end{pmatrix},
\]
and, finally, set $f(x) = \Vert A g(x) \Vert_{1}$.
\paragraph*{}Notice that if $\xi_1, \xi_2, \ldots, \xi_n$ are independently sampled from the standard normal distribution, then each $\Phi(\xi_i)$ is distributed uniformly on $[0,1]$.  Therefore $g(\xi_1, \ldots, \xi_n)$ has the same distribution as $X$, and so $F := f(\xi_1, \ldots, \xi_n)$ has the same distribution as $\Vert AX \Vert_1$.
\paragraph*{}Now let $x, y \in \R^n$ be arbitrary.  Then we have
\begin{eqnarray*}
|f(x) - f(y)| &=& \Big | \Vert A g(x) \Vert_{1} - \Vert A g(y) \Vert_1 \Big | \leq \Vert A g(x) - A g(y) \Vert_{1} \leq \sqrt{n} \Vert A (g(x) - g(y)) \Vert_{2}\\
&\leq& \sqrt{n} \Vert A \Vert_2 \Vert g(x) - g(y) \Vert_{2} \leq \sqrt{n} \Vert g(x) - g(y) \Vert_{2}.
\end{eqnarray*}
Using the fact that $| e^{i \alpha} - 1 | \leq | \alpha|$ for all $\alpha \in \R$, we further bound the above by
\begin{eqnarray*}
\Vert g(x) - g(y) \Vert_{2} ^2 &=& \sum_{j=1} ^{n} |e^{2 \pi i \Phi(x_j)} - e^{2 \pi i \Phi(y_j)}|^2  =\sum_{j=1} ^{n} |e^{2 \pi i (\Phi(x_j) - \Phi(y_j))} - 1|^2\\
&\leq& (2 \pi)^2 \sum_{j=1} ^{n} |\Phi(x_j) - \Phi(y_j)|^2 \leq 2 \pi \sum_{j=1} ^{n} |x_j - y_j|^2 = 2 \pi \Vert x - y \Vert_2 ^{2}.
\end{eqnarray*}
Thus, $|f(x) - f(y) | \leq \sqrt{2 \pi n} \Vert x - y \Vert_2$, and appealing to \myRef{Theorem \ref{talagrand gaussian theorem}} with $\sigma = \sqrt{2 \pi n}$ yields
\[
\pr(\Vert AX \Vert_1 > \E[\Vert A X \Vert_1] + tn) = \pr(F > \E[F] + tn) \leq e^{- 2 (n t)^2 / (\pi \sqrt{2 \pi n})^2} = e^{-  n t^2 / \pi ^3 }. \qedhere
\]
\end{proof}

\subsection*{Finishing the proof for $\K = \C$}

\begin{proposition}\label{perm bound in terms of mean}
Let $\Vert A \Vert_2 \leq 1$ and $X \in \C^n$ be as in \myRef{Section \ref{section set-up}}.  If $\E[\Vert A X \Vert_1 /n] = \mu$, then
\[
\E[(\Vert AX \Vert_1 /n)^n] \leq 2 \exp [-3 n (1-\mu)^2 / 100].
\]
\end{proposition}
\begin{proof}
Let $L = t \mu + (1-t)$ with $t \in [0, 1]$ to be determined.  Since $0 \leq \Vert AX \Vert_1 / n \leq 1$, we have (appealing to \myRef{Proposition \ref{our concentration}} for the last inequality)
\begin{eqnarray*}
\E[(\Vert AX \Vert_1 / n)^n] & \leq & L^n + \pr(\Vert AX \Vert_1/n > L)\\
&\leq& \exp[-n(1-L)] + \pr(\Vert AX \Vert_1/n - \mu > (1-t)(1 - \mu))\\
&\leq& \exp[-nt(1-\mu)] + \exp[-n (1-t)^2 (1-\mu)^2 / \pi^3 ],
\end{eqnarray*}
We now take $2t(1-\mu) = \pi^3+2 -2 \mu - \pi^{3/2} \sqrt{\pi^3 + 4-4 \mu} $ (for which $t$ does lie in the interval $[0,1]$), so as to make the exponents equal.  For this $t$, we obtain
\[
\E[(\Vert AX \Vert_1/n)^n] \leq 2\exp \Big [-n (2 \mu + \pi^{3/2} \sqrt{\pi^3 + 4-4 \mu}- \pi^3-2)/2 \Big ].
\]
Then appealing to the Taylor series at $\mu =1$, we see that for all $\mu \in [0,1]$,
\[
\dfrac{2 \mu + \pi^{3/2} \sqrt{\pi^3 + 4-4\mu}- \pi^3-2}{2} \geq \dfrac{(1-\mu)^2}{\pi^3} - \dfrac{2(1-\mu)^3}{ \pi^6} \geq (1-\mu)^2 \left( \dfrac{1}{\pi^3} - \dfrac{2}{\pi^6} \right) \geq \dfrac{3 (1- \mu)^2}{100}. \qedhere
\]
\end{proof}

\paragraph*{}We then readily obtain \myRef{Theorem \ref{main theorem}} simply by combining \myRef{Propositions \ref{mean bound}} and \myRef{\ref{perm bound in terms of mean}} and using the fact that if $\Vert A \Vert_2 \leq 1$, then $0 \leq h_{\infty} (A) \leq h_{2} (A) \leq 1$.

\section{Proof of \myRef{Theorem \ref{main theorem real case}} (better results for $\K = \R$)}\label{section real}
For matrices over $\R$, our general strategy is the same as before, but we first partition the rows of $A$ into those that contain `big' entries and those that do not.  We show that the contribution due to rows with large entries has small variance, and although the rows without large entries may each contribute something of high variance, we benefit from the fact that there simply aren't that many such rows.  In this way, we are able to obtain better concentration of $\Vert AX \Vert_1$ about its mean, which in turn gives a better bound on $\perm{A}$.

\paragraph*{}We are not sure exactly how to adapt this argument when $\K = \C$, although we admittedly didn't try very hard to do so.  We feel confident (especially in light of \myRef{Theorem \ref{main theorem real case}}) that \myRef{Theorem \ref{main theorem}} can be improved, but we do not think that \myRef{Theorem \ref{main theorem real case}} is best possible either (which is why we haven't worried so much about extending it to $\K = \C$).  See \myRef{Section \ref{section conclusion}} for a discussion of several related conjectures (some perhaps more true than others) and open problems.

\subsection*{Set-up for the real-valued case}
As in \myRef{Section \ref{section set-up}}, we let $A$ be an $n \times n$ matrix over $\R$ with $\Vert A \Vert_{2} \leq 1$.  Define $t = 1 - h_{\infty} (A)$.  Then to prove \myRef{Theorem \ref{main theorem real case}}, our goal is to show
\[
|\perm{A}| \leq (n+6) \exp[- \sqrt{nt} / 400].
\]

\paragraph*{}Let $\varepsilon > 0$ and $1/10 > \lambda > 0$ be parameters to be determined (we will end up choosing $\varepsilon = t / 10$ and $\lambda = 64 / \sqrt{nt}$).  We now partition the rows of $A$ into ``big rows" (those containing an element of absolute value at least $1 - \lambda$) and ``small rows" (the rest).  Suppose there are $b$ big rows and $l = n-b$ small rows.  Recall that because $\Vert A \Vert_{2} \leq 1$, each row and column of $A$ has $l_2$-norm at most $1$.  Thus, `large' entries (those of absolute value at least $1 - \lambda$) must appear in different rows and columns.  By multiplying $A$ by appropriate permutation matrices and the appropriate $\pm 1$-diagonal matrix (which changes neither the norm, nor the absolute value of the permanent, nor the values of $t, b,$ or $l$), we can assume $A$ is of the form:
\[
A = \left(
\begin{array}{c}
B \\
L
\end{array}
\right),
\]
where $B$ is a $b \times n$ matrix, the $(i,i)$-entries of $B$ are all positive with size at least $1-\lambda$, and all the rest of the entries in $A$ have absolute value less than $1 - \lambda$.  For convenience, we will assume $b > 0$ and $l > 0$, for if not, our same argument would apply with only superficial alterations.

\paragraph*{}We recall our earlier set-up as in the complex-case (but with $X \in \R^n$ now uniformly distributed over $\{-1, 1 \}^n$).  Then for all $\tilde{\mu}_B, \tilde{\mu}_L \geq 0$, we have
\begin{equation}
\begin{split}
|\perm{A}| &\leq \E_{X} \left[ \left( \dfrac{\Vert AX \Vert_{1}}{n} \right) ^n \right] = \E_{X} \left[ \left( \dfrac{\Vert LX \Vert_{1} + \Vert BX \Vert_{1}}{n} \right) ^n \right]\\
&\leq \left(\dfrac{\tilde{\mu}_{L} + \tilde{\mu}_{B}}{n}  + 2 \varepsilon \right) ^{n} + \pr \left(\Vert LX \Vert_{1} \geq \tilde{\mu}_{L} + \varepsilon n \right) + \pr \left(\Vert BX \Vert_{1} \geq \tilde{\mu}_{B} + \varepsilon n \right),
\end{split}\label{template bound}
\end{equation}
where (as before) the last inequality is justified by the fact that the random variable within the expected value is bounded above by $1$.
\paragraph*{}We choose
\begin{eqnarray*}
\tilde{\mu}_B &=& \sum_{i = 1} ^{b} \left[ \sqrt{\dfrac{2}{\pi}} + \left(1 - \sqrt{\dfrac{2}{\pi}} \right) \Vert r_i \Vert _\infty \right] = \sum_{i = 1} ^{b} \left[1 - \left(1 - \sqrt{\dfrac{2}{\pi}} \right) (1 - \Vert r_i \Vert _\infty) \right], \qquad \text{and}\\
\tilde{\mu}_L &=& \sum_{i > b} ^{n} \left[ \sqrt{\dfrac{2}{\pi}} + \left(1 - \sqrt{\dfrac{2}{\pi}} \right) \Vert r_i \Vert _\infty \right] = \sum_{i > b} ^{n} \left[1 - \left(1 - \sqrt{\dfrac{2}{\pi}} \right) (1 - \Vert r_i \Vert _\infty) \right],
\end{eqnarray*}
where (again) $r_i$ is the $i^{\text{th}}$ row of $A$ (note, $\Vert r_i \Vert_{\infty} = b_{i,i}$ for all $i \leq b$).  Then by \myRef{Theorem \ref{konig_ineq}} (this time with $\K = \R$), we have $\tilde{\mu}_L \geq \E [ \Vert LX \Vert_1 ]$ and $\tilde{\mu}_B \geq \E[ \Vert BX \Vert_1 ]$, and by the definitions
\begin{equation}\label{t and mu}
\dfrac{\tilde{\mu}_L + \tilde{\mu}_B}{n} = 1 - \left(1 - \sqrt{\dfrac{2}{\pi}} \right) \dfrac{1}{n} \sum_{i=1} ^{n} \bigg( 1 - \Vert r_i \Vert_\infty \bigg) = 1 - \left(1 - \sqrt{\dfrac{2}{\pi}} \right) t.
\end{equation}

\paragraph*{}To take advantage of \eqref{template bound}, we need only exhibit concentration bounds for $\Vert LX \Vert_1$ and $\Vert BX \Vert_1$.

\subsection*{Concentration of $\Vert LX \Vert_1$}
To show concentration of $\Vert LX \Vert_1$ about its mean, we will again apply a version of Talagrand's inequality (but this time suited for the discrete distribution over $\{-1, 1\}^n$).  Instead of showing the derivation of this from the corresponding general result in \cite{talagrand} (as we did before), we will simply cite \cite{ailon}, in which the following statement appears as \myRef{Theorem 3.3}.
\begin{theorem}\label{Talagrand concentration}
Suppose $M$ is a $k \times n$ real-valued matrix such that $\Vert M \vec{x} \Vert_{1} \leq \sigma \Vert \vec{x} \Vert_2$ for all $\vec{x} \in \R^{n}$.  Let $\xi \in \R^n$ be chosen uniformly from $\{-1,1\}^n$, and let $m$ be a median of $\Vert M \xi \Vert_1$.  Then for all $\gamma \geq 0$, we have $\pr ( | \Vert M \xi \Vert_{1} - m| > \gamma) \leq 4 e^{-\gamma^2 / (8 \sigma^2)}$.
\end{theorem}
\begin{lemma}\label{bound on LX deviation}
With notation as before, if $\varepsilon n \geq 16 \sqrt{nt \log (n) / \lambda}$, then
\[
\pr \left(\Vert LX \Vert_{1} \geq \tilde{\mu}_{L} + \varepsilon n \right) \leq 4 \exp \left[ \dfrac{-\varepsilon^2 n \lambda}{32 t} \right].
\]
\end{lemma}
\begin{proof}
Note that for all $\vec{x} \in \R^n$, we have $\Vert L \vec{x} \Vert_{1} \leq \sqrt{l} \Vert L \vec{x} \Vert_{2} \leq \sqrt{l} \Vert A \vec{x} \Vert_{2} \leq \sqrt{l} \Vert \vec{x} \Vert_2$.  Thus, if $m$ is a median of $\Vert L X \Vert_1$, then by \myRef{Theorem \ref{Talagrand concentration}}, we have
\begin{equation}
\pr(| \Vert LX \Vert_1 - m| > \gamma ) \leq 4 e^{-\gamma^2 / (8 l)}. \label{median bound on L}
\end{equation}
From this, we see that $\Vert LX \Vert_1$ is tightly concentrated about its \textit{median}.  However, this also implies
\begin{equation}\label{median is close to mean}
m \leq \E[ \Vert L X \Vert_1 ] + 8 \sqrt{l \log n},
\end{equation}
since otherwise, we would have
\begin{eqnarray*}
\E[ \Vert L X \Vert_1 ] &\geq& \left( \E[ \Vert L X \Vert_1 ] + 4 \sqrt{l \log n} \right) \cdot \pr \left(| \Vert LX \Vert_1 - m| \leq 4 \sqrt{l \log n} \right)\\
&\geq& \left( \E[ \Vert L X \Vert_1 ] + 4 \sqrt{l \log n} \right) \cdot (1 - 4 / n^2 )\\
&=& \E[ \Vert L X \Vert_1 ] + 4 \sqrt{l \log n} - \left( \E[ \Vert L X \Vert_1 ] + 4 \sqrt{l \log n} \right ) \cdot 4/ n^2.
\end{eqnarray*}
And subtracting  $\E[ \Vert LX \Vert_1 ]$ from both sides and rearranging, we would obtain
\[
n^2  \leq 4 + \dfrac{\E[ \Vert L X \Vert_1 ]}{\sqrt{l \log n}} \leq 4 + \dfrac{n}{\sqrt{\log n}},
\]
which is a contradiction if $n > 2$ (whereas for $n \leq 2$, the desired bound on $m$ is implied by $m \leq n$ [not that it matters]).  Therefore, appealing to \eqref{median is close to mean}, we have
\[
\pr \left(\Vert LX \Vert_{1} \geq \tilde{\mu}_{L} + \varepsilon n \right) \leq \pr \left(\Vert LX \Vert_{1} \geq \E[ \Vert LX \Vert_1 ] + \varepsilon n \right) \leq \pr \left(\Vert LX \Vert_{1} \geq m + \varepsilon n - 8 \sqrt{ l \log n} \right).
\]
Furthermore, if $\varepsilon n \geq 16 \sqrt{l \log n}$, then we can combine this with \eqref{median bound on L} to obtain
\begin{equation}
\text{if $\varepsilon n \geq 16 \sqrt{l \log n}$, then} \qquad \pr \left(\Vert LX \Vert_{1} \geq \tilde{\mu}_{L} + \varepsilon n \right) \leq 4 \exp \left[ \dfrac{-\varepsilon^2 n^2}{32 l} \right].\label{original bound on LX deviation}
\end{equation}
Finally, since $nt \geq \sum_{i=b+1} ^{n} (1 - \Vert r_i \Vert_{\infty}) \geq l \lambda$, we know $l \leq nt / \lambda$, completing the proof by \eqref{original bound on LX deviation}.
\end{proof}

\subsection*{Concentration of $\Vert BX \Vert_1$}
We now focus on getting an upper bound on $\pr(\Vert BX \Vert_1 \geq \tilde{\mu}_B + \varepsilon n)$.  We first recall the following classical concentration result.
\begin{proposition}[Hoeffding's inequality]\label{hoeffding}
Let $a_1, \ldots , a_k$ be real numbers (not all of which are $0$), and let $\xi_1, \xi_2, \ldots , \xi_k$ be independent each distributed uniformly on $\{-1, 1\}$.  Then for all $\gamma \geq 0$,
\[
\pr \left( \sum_{i=1} ^{k} a_i \xi_i \geq \gamma \right) \leq \exp \left[ \dfrac{- \gamma^2}{2 \sum_{i=1} ^{k} a_i ^2} \right].
\]
\end{proposition}

Let $\otherB = \left(
\begin{array}{c}
B \\
0
\end{array}
\right)$ be the $n \times n$ matrix whose first $b$ rows are given by $B$ and the rest are $0$.  Our key step here is replacing $\Vert BX \Vert_1$ with $\langle X, \otherB X \rangle$, via the following lemma\footnote{Extending this step is the main obstacle to applying the present argument when $\K = \C$.}.
\begin{lemma}\label{BX is close to inner product}
With notation as before, if $\lambda < 0.1$ then
\[
\pr ( \Vert BX \Vert_1  \geq \tilde{\mu}_B + \varepsilon n) \leq \pr ( \langle X, \otherB X \rangle  \geq \tilde{\mu}_B + \varepsilon n) + n e^{-1/(5 \lambda)}.
\]
\end{lemma}
\begin{proof}
It suffices to show $\pr ( \Vert BX \Vert_1  \neq \langle X, \otherB X \rangle ) \leq n e^{-1/(5 \lambda)}$.  The idea is that since each row of $B$ is dominated by a single large entry (namely $b_{i,i}$), each entry of $BX$ is a random sum dominated by a single large term (namely $X_i b_{i,i}$).  Thus, it is very unlikely that any entry of $BX$ would have a different sign than $X_i b_{i,i}$.  This is made rigorous as follows.
\paragraph*{}Recall that we ordered the columns of $B$ so that the $(i,i)$-entry is the largest in its row, and that $b_{i,i} \geq 1- \lambda$.  Letting $Y_i$ be the $i^{\text{th}}$ coordinate of $B X$, we have, by a simple union bound,
\[
\pr ( \Vert BX \Vert_1  \neq \langle X, \otherB X \rangle ) \leq  \sum_{i=1} ^{b} \pr( |Y_i|  \neq X_i Y_i ) = \sum_{i=1} ^{b} \pr(X_i Y_i < 0) = \sum_{i=1} ^{b} \pr \left(\sum_{j=1} ^{n} X_i X_j b_{i,j} < 0 \right).
\]
Using the fact that for any given $i$, the random vector $(X_i X_j)_{j \neq i}$ has the same joint distribution as $(X_j)_{j \neq i}$ (and that $X_i ^2 = 1$), we obtain by \myRef{Proposition \ref{hoeffding}}
\[
\sum_{i=1} ^{b} \pr \left( \sum_{j=1} ^{n} X_i X_j b_{i,j} < 0 \right) = \sum_{i=1} ^{b} \pr \left( b_{i,i} < \sum_{j \neq i} ^{n} X_j b_{i,j} \right) \leq \sum_{i=1} ^{b} \exp \left[ \dfrac{- b_{i,i} ^2}{2 \sum_{i\neq j} b_{i,j} ^2} \right].
\]
Since $b_{i,i} \geq 1-\lambda$ and $\sum_{j} b_{i,j} ^2 \leq 1$, this in turn is bounded by
\[
\sum_{i=1} ^{b} \exp \left[ \dfrac{- b_{i,i} ^2}{2 \sum_{i\neq j} b_{i,j} ^2} \right] \leq n \exp \left[ \dfrac{- (1-\lambda) ^2}{2 (1- (1-\lambda)^2)} \right] \leq n e^{-1/(5 \lambda)},
\]
where the last inequality is justified because $0 < \lambda < 0.1$.
\end{proof}

\paragraph*{}We can now exploit the fact that $\langle X, \otherB X \rangle$ is a degree two polynomial over $\{-1, 1\}^n$, allowing us to use any of a variety of concentration inequalities.  We will use an inequality of Bonami \cite{bonami}, which was the first \textit{hypercontractivity inequality} of its type.  A detailed exposition of such results can be found in chapter 9 of O'Donnell's book \cite{odonnell}, and a comparison of this to more recent polynomial concentration inequalities can be found in \cite{schudy}.

\begin{theorem}[Bonami \cite{bonami}, 1970]\label{bonami theorem}
Let $F : \R^{n} \to \R$ be a degree $k$ polynomial, and consider the random variable $Z = F(\xi_1, \xi_2, \ldots , \xi_n)$, where the $\xi_i$ are independent with each distributed uniformly over $\{-1, 1\}$.  Then for all $q \geq 2$, we have $\E[|Z|^{q}] \leq \left( (q-1)^{k} \E[Z^2] \right)^{q/2}.$
\end{theorem}
\begin{lemma}\label{bound on inner product}
With notation as before, if $\varepsilon n \geq 4 e \sqrt{nt}$, then
\[
\pr ( \langle X, \otherB X \rangle  \geq \tilde{\mu}_B + \varepsilon n) \leq \exp \left( \dfrac{- \varepsilon n}{2e \sqrt{nt}} \right).
\]
\end{lemma}
\begin{proof}
For $\vec{x} \in \R^n$, define $F(x_1, x_2, \ldots , x_n) = \langle \vec{x}, \otherB \vec{x} \rangle - \displaystyle \sum_{i=1} ^{b} b_{i,i}$, and define the random variable $Z = F(X_1, \ldots , X_n)$.  Then $\pr ( \langle X, \otherB X \rangle  \geq \tilde{\mu}_B + \varepsilon n) \leq \pr( Z  \geq \varepsilon n)$, since\footnote{In fact, we could have simply taken $\tilde{\mu}_B = \sum_{i \leq b} b_{i,i}$, but we chose instead to define it similarly to $\tilde{\mu}_L$, a change which only affects the constants in our end result.} $\tilde{\mu}_{B} \geq  \sum_{i\leq b} b_{i,i}$.  Now $F(x_1, x_2, \ldots, x_n)$ is a degree $2$ polynomial, and moreover, by expanding out the sums and using the fact that terms such as $\E[X_i X_j]$ vanish when $i \neq j$, we obtain
\begin{eqnarray*}
\E[Z^2] &=& \E \left[ \left( \sum_{i=1} ^{b} \left[-b_{i,i} + \sum_{j = 1} ^{b} X_i X_j b_{i,j} \right] + \sum_{i=1} ^{b} \sum_{j = b+1} ^{n} X_i X_j b_{i,j} \right)^2 \right]\\
&=& \E \left[ \left( \sum_{i=1} ^{b} \left[-b_{i,i} + \sum_{j = 1} ^{b} X_i X_j b_{i,j} \right] \right)^2 \right] + \E \left[ \left( \sum_{i=1} ^{b} \sum_{j = b+1} ^{n} X_i X_j b_{i,j} \right)^2 \right]\\
&=& \sum_{i=1} ^{b} \sum_{j < i} (b_{i,j} + b_{j,i})^2 + \sum_{i=1} ^{b} \sum_{j=b+1} ^{n} b_{i,j} ^2 \leq 2 \sum_{i=1} ^{b} \sum_{j < i} (b_{i,j} ^2 + b_{j,i} ^2) + 2\sum_{i=1} ^{b} \sum_{j=b+1} ^{n} b_{i,j} ^2\\
&=& 2 \sum_{i=1} ^{b} \left( -b_{i,i}^2 + \sum_{j=1} ^{n} b_{i,j} ^2 \right) \leq 2 \sum_{i=1} ^{b} (1 - b_{i,i} ^2) \leq 4 \sum_{i=1} ^{b} (1- b_{i,i}) \leq 4 nt.
\end{eqnarray*}
Applying \myRef{Theorem \ref{bonami theorem}} with $q = \varepsilon n / (2e \sqrt{nt})$---which is valid since by hypothesis this ratio is at least 2---together with Markov's inequality, we obtain
\[
\pr( Z  \geq \varepsilon n) \leq \pr ( |Z|^{q} \geq (\varepsilon n)^q ) \leq \dfrac{\E[|Z|^q]}{(\varepsilon n)^{q}} \leq \left( \dfrac{(q-1) 2 \sqrt{nt}}{\varepsilon n} \right)^{q} \leq \exp \left( \dfrac{- \varepsilon n}{2e \sqrt{nt}} \right). \qedhere
\]
\end{proof}

\subsection*{Finishing the proof for $\K = \R$}
We now need to pick $\varepsilon$ and $\lambda$ to optimize the tradeoffs between our various upper bounds. We need the assumptions of \myRef{Lemmas \ref{bound on LX deviation}}, \myRef{\ref{BX is close to inner product}}, and \myRef{\ref{bound on inner product}}---namely (i) $\varepsilon n \geq 16 \sqrt{nt \log(n) / \lambda}$, (ii) $\lambda < 0.1$, and (iii) $\varepsilon n \geq 4 e \sqrt{nt}$---in which case we can combine these lemmas with \eqref{template bound} and \eqref{t and mu} to obtain
\begin{eqnarray*}
|\perm{A}| &\leq& \left(2 \varepsilon + \dfrac{\tilde{\mu}_{L} + \tilde{\mu}_{B}}{n} \right) ^{n} + \pr \left(\Vert LX \Vert_{1} \geq \tilde{\mu}_{L} + \varepsilon n \right) + \pr \left(\Vert BX \Vert_{1} \geq \tilde{\mu}_{B} + \varepsilon n \right)\\
&\leq& \left(2 \varepsilon + 1 - \left(1 - \sqrt{\dfrac{2}{\pi}} \right) t \right) ^{n} + 4 \exp \left[ \dfrac{-\varepsilon^2 n \lambda}{32 t} \right] + n e^{-1/(5 \lambda)} + \exp \left( \dfrac{- \varepsilon n}{2e \sqrt{nt}} \right).
\end{eqnarray*}

\paragraph*{}We will take $\varepsilon = t/10$ and $\lambda = 64 / \sqrt{nt}$, for which we claim that conditions (i), (ii), and (iii) are satisfied.  Note that since our goal is to show $|\perm{A}| \leq (n+6) \exp[- \sqrt{nt} / 400]$, we may assume $\sqrt{nt}/\log(n+6) \geq 400$ (or the bound we are trying for is worse than the trivial bound of $1$) (of course, in any case we are really more interested in large $n$).  Notice that with $\varepsilon$ and $\lambda$ as above:
\begin{itemize}
\item[(i)] $\varepsilon n \geq 16 \sqrt{nt \log(n) / \lambda}$ is equivalent to $\sqrt{nt} \geq 400 \log n$;
\item[(ii)] $\lambda < 0.1$ is equivalent to $\sqrt{nt} > 640$; and
\item[(iii)] $\varepsilon n  \geq 4e \sqrt{nt}$ is equivalent to $\sqrt{nt} \geq 40 e$.
\end{itemize}
Thus, these choices of $\lambda$ and $\varepsilon$ allow us to appeal to the aforementioned results, obtaining
\begin{eqnarray*}
|\perm{A}| &\leq& \left(2 \varepsilon + 1 - \left(1 - \sqrt{\dfrac{2}{\pi}} \right) t \right) ^{n} + 4 \exp \left[ \dfrac{-\varepsilon^2 n \lambda}{32 t} \right] + n e^{-1/(5 \lambda)} + \exp \left( \dfrac{- \varepsilon n}{2e \sqrt{nt}} \right)\\
&\leq& \exp \left[-nt \left(1 - \sqrt{2/\pi} - 0.2 \right)  \right]  + 4 \exp \left[ \dfrac{-\sqrt{nt}}{50} \right] + n \exp \left[-\dfrac{\sqrt{nt}}{320} \right] + \exp \left[ \dfrac{- \sqrt{nt}}{20 e} \right]\\
&\leq& (n+6) \exp \left[\dfrac{-\sqrt{nt}}{400} \right],
\end{eqnarray*}
which completes the proof of \myRef{Theorem \ref{main theorem real case}}.

\section{Conclusion}\label{section conclusion}
Our biggest (and most natural) open question concerns the optimality of our main results.  Namely, a proof of \myRef{Conjecture \ref{main conjecture}} as stated in \myRef{Section \ref{section intro}} would be very interesting.  The main barrier preventing us from proving this conjecture is our reliance on Talagrand's inequality.  For $\K = \R$, we partially mitigated the cost of using this inequality via \myRef{Lemma \ref{bound on LX deviation}}, but the application of \myRef{Theorem \ref{Talagrand concentration}} was still a crucial (though not the only) bottleneck.  Our argument could conceivably be pushed further either by a more careful analysis that better uses \eqref{original bound on LX deviation} or by a more nuanced argument that splits the matrix $A$ into more than two pieces.

\paragraph*{}One could also try to avoid using Talagrand's inequality altogether.  It is possible that some stronger inequality could replace it (by taking advantage of some aspects particular to our situation), but a more likely ``quick fix" of this sort would be a more direct estimate of $\E[(\Vert A X \Vert_1 /n)^n]$ (in the real case, $AX$ is simply a vector-valued Rademacher sum, which is a well-studied random variable).  On the other hand, it could be that the convexity bounds on the Glynn estimator already give away too much to recover anything stronger than what we have.

\paragraph*{}An entirely different approach would be to determine among matrices with given norm and $h_{\infty}$, which ones maximize $| \perm{A}|$ (it does not seem impossible that this maximum is always attained by a circulant matrix with all real entries).  A characterization of these extremal matrices would certainly be very appealing, and one might hope that thinking along these lines would suggest a more combinatorial approach.

\paragraph*{}As far as \myRef{Question B} is concerned, we feel that there is still more to be said beyond the present results.  Namely, our results only provide a necessary condition for a matrix to have a large permanent (i.e., $h_{\infty}$ must be large).  But there is no clean converse to this statement; consider for example a diagonal matrix with most of its diagonal entries equal to 1 except for one of them equal to 0 (this has large $h_{\infty}$ and permanent $0$).  To continue the spirit of the question, we state the following variation of \myRef{Question B} (essentially echoing a question of \cite{aaronson}):
\paragraph*{Problem B$'$:}Find a (deterministic) polynomial-time algorithm that takes an $n \times n$ matrix $A$ of norm $1$ and decides whether $|\perm{A}| < n^{-100}$ or $|\perm{A}| > n^{-10}$ (with the understanding that the input matrix will satisfy one of these inequalities).

\paragraph*{}We attempted this along the following lines: ``if the matrix has large permanent, it must have many rows each of which is dominated by a single large entry.  If the matrix is of this form, then [heuristic] hopefully that means the permanent is dominated by terms that use at least most of these large entries.  Since there are so many large entries, we can efficiently compute the exact contribution of these dominant terms."  However, our current results do not allow us to conclude that there are enough rows with large entries (we would like all but about $\log n$ of the rows but are limited to all but about $\log^2 n$ when $\K = \R$ and $\sqrt{n \log n}$ when $\K = \C$).  And in fact, even if we could improve our result to the conjectured (and best possible) bound mentioned above, we still do not quite see how to make this heuristic argument yield a polynomial-time algorithm.  We should note that Gurvits \cite{gurvits} found a \textit{randomized} algorithm accomplishing the goal of \myRef{Problem B$'$}, and in the deterministic setting, progress towards \myRef{Problem B$'$} was made in \cite{aaronson} which gives an algorithm in the case that the entries of $A$ are non-negative.

\subsection*{Further remarks}
\begin{itemize}
\item We note that there is a lot of freedom in choosing the random variable $X \in \K^{n}$ for the Glynn estimator ($X$ just needs to have independent components each satisfying $\E[X_i] =0$ and $\E[|X_i|^2] = 1$).  For example, when $\K = \R$, it is tempting to replace $X \in \R^n$ with an $n$-dimensional Gaussian and bound the Glynn estimator by something like
\[
| \perm{A} | = \left| \E \left[ \prod_{i} X_i Y_i \right] \right| \leq \E \left[ \prod_{i} |X_i Y_i| \right] \leq \E \left[ \left( \dfrac{1}{n} \sum_{i} |X_i Y_i| \right)^n \right].
\]
But even if $A$ is the identity matrix this is already (exponentially) larger than $1$, which illustrates the difficulty with this approach.
\item Via an entirely different method, we were also able to get an upper bound on the permanent for matrices having only non-negative real entries by appealing to the results of \cite{gurvitsSam}.  Unfortunately, the bound we obtained is strictly weaker than the results of the present paper, so it is omitted.
\end{itemize}
\textbf{Acknowledgement:} We thank Hoi Nguyen for introducing us to this problem and sharing \cite{nguyenPrivate}.

\bibliography{MyBib}
\bibliographystyle{plain}

\vspace*{.5 in}
\noindent Department of Mathematics\\
Rutgers University\\
Piscataway, NJ 08854\\
\texttt{rkb73@math.rutgers.edu}\\
\texttt{prd41@math.rutgers.edu}

\end{document}